\documentclass[12pt]{article}
\usepackage{amsthm}
\usepackage{amssymb}
\usepackage{amsmath}
\usepackage{amsfonts}
\usepackage[T1]{fontenc}
 \usepackage{graphicx}
\usepackage{color}

\usepackage[ruled,vlined,noend]{algorithm2e}
\newtheorem{thm}{Theorem}[section]
\newtheorem{pro}{Proposition}[section]

\newtheorem{defin}{Definition}[section]

 \newcommand{\R}{\mathbb{R}}
  \newcommand{\card}{{\rm{card}}}

  \SetNlSty{textbf}{Step }{:}
  \SetAlgoNlRelativeSize{0}
  \SetKwBlock{Init}{Initialisation}{end}
  \SetKwBlock{Main}{Main Step}{end}
  \DontPrintSemicolon
  \SetKwFor{Case}{Case}{}{end}

\begin{document}

\title{Alternance Theorems and Chebyshev Splines Approximation}

\author{Jean-Pierre Crouzeix \thanks{Universit\'e Clermont Auvergne, Clermont Ferrand,
France,  jp.crouzeix@isima.fr, 24 Avenue des Landais, 63170 Aubière, France, Phone: +33 4 73 40 63 63}
\and Nadezda Sukhorukova \thanks{Corresponding author, Swinburne University of Technology, Melbourne, Australia 
              and  Federation University Australia, nsukhorukova@swin.edu.au, Swinburne University of Technology
PO Box 218
Hawthorn, Victoria, 3122
Australia, Phone: +61 3 9214 8455
Fax: +61 3 9214 8264}
               \and Julien Ugon \thanks{Deakin University, Melbourne, Australia and Federation University Australia, Ballarat, Australia,  julien.ugon@deakin.edu.au, 221 Burwood Highway,
Burwood VIC 3125
Australia,  Phone
+61 3 9244 6100}}


\maketitle
 
 \begin{abstract}
 One of  the purposes in  this  paper  is to provide a better understanding of  the alternance property  which occurs  in Chebyshev polynomial approximation  and piecewise polynomial approximation problems. In the first part of this paper, we propose an original approach to obtain new proofs of the well known necessary and sufficient optimality conditions. There are two main advantages of this approach. First of all, the proofs are much  simpler and easier to understand than the existing proofs. Second, these proofs are constructive and therefore they lead  to alternative-based    algorithms that can be considered as Remez-type approximation  algorithms. In the second part of this paper, we develop new local optimality conditions for free knot polynomial spline approximation. The  proofs for free knot approximation are relying on the techniques developed in the first part of this paper. 
 
\end{abstract}

\noindent \textbf{Keywords:} Chebyshev Approximation, Polynomial Splines, Fixed and Free knots.\\
\noindent \textbf{AMS classification:} 49J52, 90C26, 41A15, 41A50.

\section{Introduction}\label{sec:one}
For a given  continuous function $f:[a,b]\rightarrow \R$, the Chebyshev polynomial approximation problem is   
$$ \min_{\pi \in \Pi_n}\,\left[\, \| \pi -f\|=\sup_{t\in [a,b]}|\,\pi(t)-f(t)\,| \,\right ],\eqno{(cpa)}$$
where $\Pi_n$ denotes the set of   polynomial functions  of degree  less or equal to~$n$.  
It is known that this problem has a unique optimal solution which is characterized by the existence of $\varepsilon \in \{ -1,1\}$ and  $n+2$ points $ t_i$   such that 
$$a\leq t_0<t_1<\cdots t_n<t_{n+1}\leq b,  \quad \|\pi-f\| =\varepsilon (-1)^i[\pi(t_i)-f(t_i)] \;\;\forall \,i .$$
It can be solved by the celebrated Remez algorithm  (1959) \cite{rem}.

This problem can be also formulated as
$$\min _{(a,\lambda )\in \R^{n+2}}\left [ \,  \lambda   \,:\, f(t)-\sum_{i=0}^na_i t^i \leq \lambda,  \quad  \sum_{i=0}^na_i t^i -f(t) \leq \lambda \;\;  \forall \, t\in [a,b] \, \right] \eqno{(lsp)}$$
and therefore it belongs to the class of  linear semi-infinite programs.  One observes that the number  of alternating points $t_i$ corresponds  to the dimension of the space of  primal variables in $(lsp)$. The  exchange rules in the Remez algorithm  roughly  correspond   to the leaving/entering rules in the simplex algorithm running over the corresponding dual problem. However, the linear semi-infinite programming theory does not explain the alternance property which is due, as shown in this paper,  both to the continuity of $f-\pi$  and the structure of $\Pi_n$. 

Alternance conditions also appear in Chebyshev spline  approximation where $f$ is approximated by a continuous piecewise polynomial function~\cite{num1,shu,Su1}.  When the knots (points that connect polynomial pieces) are fixed the problem is no more linear but convex, when  the knots are not known, the problem is no more convex and therefore very hard to solve.  There have been several attempts to extend the results to the case of free knots polynomial spline approximation~\cite{num1,Su2,SuU2017}. The most advanced results have been obtained in~\cite{SuU2017}, where the most accurate necessary optimality conditions are obtained. Theses results are equivalent to Demyanov-Rubinov stationarity~\cite{dr1,dr2} and characterise local optimality.

In this paper, we consider exactly the same problems as the ones recently investigated by Sukhorukova and Ugon~\cite{SuU2017} and Crouzeix et al. \cite{Su4}. However, our approach and techniques are very different in their essence. The main advantages of our approach is that the proofs are easier to understand and, most importantly, the proofs are constructive. The goal of this study is also to  enhance the comprehension of   the alternance property in Chebyshev piecewise polynomial approximation problems.

In section~\ref{sec:two}, we introduce a very general   alternance result which holds for any continuous function. In the same section we also introduce a generalisation of the notion of alternating ($\beta$-alternating) which is more suitable for computational purposes. This result is applied in section~\ref{sec:three}  to the  Chebyshev (uniform) polynomial approximation and provides an alternative proof of the existence, unicity and characterization  of the optimal solution. The proof is very simple, constructive and therefore it gives rise to alternative algorithms   with  the celebrated Remez algorithm, developed for polynomial approximation. Next, the   results  of sections~\ref{sec:two}~and~\ref{sec:three} is also successfully applied in section~\ref{sec:four}  to the  fixed knots spline approximation problem. 
 
Finally, section~\ref{sec:five} treats of the free knots spline problems. The problem is no more convex so that only local optimality conditions can be obtained.

\section{Some general results}\label{sec:two}

Through the paper,  given $a,b\in \R$  with $a<b$, the norm of $f:[a,b]\rightarrow \R $ is defined by
 $$\|f\|=\max_{t\in [a,b]} | f(t)|.$$ 

\begin{thm}\label{thm:alternance}Let $f:[a,b]\rightarrow \R $ be  continuous and such that $M:= \max \,[\, f(t): a \leq t \leq b\, ]=- \min\, [\,f(t): a \leq t \leq b\,]>0$. 
\begin{enumerate}
\item
 There exist $k\geq 1$ integer, $\varepsilon  \in \{-1,1\}$,  and  $\{(t^-_i,t^+_i)\}_{i=0}^{i=k}\subset [a,b]^2$ such that 
$$a \leq t^-_0\leq  t^+_0<t^-_1\leq  t^+_1< t^-_2\leq  t^+_2<\cdots <t^-_{k}\leq    t^+_{k}\leq b,$$
$$\varepsilon M=(-1)^i f(t^-_i)= (-1)^{i}f(t^+_i), \; i=0,1, \cdots , k,   $$
 $$-M<\varepsilon  (-1)^{i}f(t)\leq M \textrm{ if } \;\;  t\in   [t^-_i, t^+_i ]  \textrm{ for some } i,$$
$$ |f(t)| <M\textrm{ for all other  } t\in [a,b].$$
The quantities $k, \varepsilon , t^-_i,t^+_i$ are uniquely defined.
\item \(\|f-p\| \geq \|f\|\) for any polynomial function \(p\) of degree at most \(k-1\). 
\item 
 Let us consider the polynomial function  of degree $k$  $$\gamma  (t)=  \varepsilon \prod _{i=0}^{k-1}\, (\xi_i -t),$$
where, for $i=0,1, \cdots , k-1$,  $\xi_i$ is arbitrarily chosen in $(t^+_i, t^-_{i+1})$.
Then, for $\lambda >0$ small  enough, it holds $\|f-\lambda \gamma  \| <\|f\| $.
\end{enumerate}
 \end{thm}
 
Below we will give a proof for a more general version of this theorem, which is better adapted for numerical purposes. Indeed, except  for some particular
functions, it is quite unrealistic to consider that the maximum of $|f(x)|$
can be reached in more than three points.

\begin{thm}\label{thm:num} Let $f:[a,b]\rightarrow \R $ be  continuous and such that 
$$M:= \max \,[\, f(t): a \leq t \leq b\, ]=- \min\, [\,f(t): a \leq t \leq b\,]>0.$$ 
Let $\beta\in (0,1]$.
\begin{enumerate}
\item  There exist $k (\beta) \geq 1 $ integer, $\varepsilon  (\beta )\in \{-1,1\}$,
$t ^-_i(\beta),\,t ^+_i(\beta) \in \R $ for $i=0,\cdots, k(\beta)$ 
such that 
$$a\leq t ^-_0(\beta)\leq t ^+_0(\beta),\quad t ^-_{k(\beta )}(\beta)\leq t ^+_{k(\beta )}(\beta)\leq b,$$
$$t ^-_i(\beta)\leq t ^+_i (\beta) <t ^-_{i+1}(\beta) \leq t ^+_{i+1}(\beta), \; i=0,1, \cdots , k(\beta )-1,$$
$$-\beta M<\varepsilon (\beta ) (-1)^{i}f(t)\leq M \;\; \forall \,  t\in   [t ^-_{i}(\beta),t ^+_{i}(\beta)] ,\; \; \forall \, i ,$$
$$  \varepsilon (\beta ) (-1)^{i}f(t^+_i (\beta ))= \varepsilon (\beta ) (-1)^{i+1}f(t^-_{i+1} (\beta ))=\beta  M, \quad   i =0, \cdots, k(\beta)-1,$$
$$  \varepsilon (\beta ) f(t^-_0 (\beta ))\geq \beta M, \quad \varepsilon (\beta ) (-1)^{k(\beta )}f(t^+_{k(\beta)} (\beta ))\geq \beta  M, $$
$$|f(t)|<\beta M \textrm{ for all other }  t\in [a,b].$$
 The quantities $\varepsilon (\beta), k(\beta),  t ^-_i(\beta ), t ^+_i(\beta ) $ are uniquely defined.
 \item \(\|f-p\| \geq \beta\|f\|\) for any polynomial function \(p\) of degree at most \(k(\beta)-1\).
 \item  Let us consider the function  
 $$\gamma _\beta  (t)=  \varepsilon  (\beta )\prod _{i=0}^{k(\beta )-1}\, (\xi_i(\beta )-t) ,$$
 where, for $i=0,1, \cdots , k(\beta)-1$,  $\xi_i$ is arbitrarily chosen in $(t^+_i(\beta ), t^-_{i+1}(\beta ))$.
 Then, $\|f-\lambda \,\gamma _\beta  \| < \|f\|$ for a suitably chosen~$\lambda >0$. 
\end{enumerate}
 \end{thm}

 \begin{proof}1)  We construct the sequence \(t_0^-(\beta),\ldots,t_{k(\beta)}^+(\beta)\) using Algorithm~\ref{algo:construction}.

   \begin{algorithm}[H]
         \caption{Construction of the alternating sequence} \label{algo:construction}
   \nlset{0} \textbf{Initialisation} Define
         \[t_0^-(\beta) = \min_{t\in [a,b]} \{t: |f(t)| \geq \beta M\} \text{ and } \varepsilon(\beta) = \tfrac{f(t_0^-(\beta))}{\vert f(t_0^-(\beta))\vert}.\]  By construction,
         \[ | f(t)|< \beta M \leq \varepsilon (\beta) f(t_0^-(\beta)) \;  \textrm{ if } \;a\leq t< t^-_0(\beta).   \]
         Set \(i:=1\).\;
   
  \While{\( \{t: (-1)^{i-1} \varepsilon (\beta )f(t) \leq  -\beta M\} \neq \emptyset\)}{
             \nlset{i} \text{Let } \[\left\{\begin{array}{ll} t_i^-(\beta) &= \min_{t\in [t_{i-1}^-(\beta),b]} \{t: (-1)^{i-1} \varepsilon (\beta )f(t) \leq  -\beta M\}\text{ and } \\
             t_{i-1}^+(\beta) &= \max_{t\in [t_{i-1}^-(\beta),t_i^-(\beta)]} \{t: (-1)^{i-1} \varepsilon (\beta )f(t) \geq  \beta M\}.\end{array}\right.\]

 Such $t^+_{i-1}(\beta),t^-_{i}(\beta) $ exist and $t^-_{i-1}(\beta)\leq t^+_{i-1}(\beta)<  t^-_{i}(\beta)\leq b$.
By construction, 
\[ \begin{array}{l}
        \beta   M=(-1)^{i} \varepsilon (\beta ) f(t^-_{i}(\beta))=   (-1)^{i-1} \varepsilon (\beta )f(t^+_{i-1}(\beta)),\\
-\beta M<\varepsilon (\beta ) (-1)^{i-1} f(t)\leq M  \textrm{ if } \;t^-_{i-1}(\beta)\leq t\leq t^+_{i-1}(\beta),\\
| f(t) |<  \beta M  \; \textrm{ if } \;t^+_{i-1}(\beta)< t<t^-_{i}(\beta).
\end{array} \]
}
\nlset{k} Set \(k(\beta)=i\). \\Let 
\(\displaystyle t_{k(\beta)}^+(\beta) = \max_{t\in [t_{k-1}^-(\beta),b]} \{t:  (-1)^{k(\beta)} \varepsilon(\beta)f(t) \geq  \beta M\}\).\\
By construction, 
\[ \begin{array}{l}
\beta M=(-1)^{k(\beta)} \varepsilon (\beta )f(t^-_{k(\beta)}(\beta))\leq (-1)^{k(\beta)} \varepsilon (\beta )f(t^+_{k(\beta)}(\beta)),\\
-\beta M<\varepsilon (\beta) (-1)^{k(\beta )} f(t)\leq M  \textrm{ if } \;t^-_{k(\beta)}(\beta)\leq t\leq t^+_{k(\beta)}(\beta), \\
| f(t) |<  \beta M  \; \textrm{ if } \;t^+_{k(\beta)}(\beta)< t\leq b.
\end{array} \]
     \end{algorithm}

 It remains to prove that the construction ends after a finite number of steps. We proceed by contradiction. If not, we have built a strictly increasing sequence $\{t^-_i(\beta)\}_i\subset [a,b]$. This sequence converges to some $\bar t \in [a,b]$. Since $f$ is continuous, $f(t^-_i(\beta))$ converges  to $f(\bar t)$. But $f(t^-_i(\beta))$ equals  either $\beta M$ or $-\beta M$ depending on the parity of $i$. This is not possible.

 2)  Assume  that  \(p \in \Pi_{k(\beta)-1}\) exists  such that \(\|f-p\|< \beta \, \|f\| =\beta M\). Then, for each $i=0,\cdots, k(\beta)$
 $$ \beta M >  \varepsilon (\beta )(-1)^i [f(t_{i}^- (\beta )-p(t_{i}^- (\beta ))] \geq  
 \beta M-\varepsilon(\beta) (-1)^i p(t_i^-(\beta))  .$$ 
 Hence, \( p(t_i^-(\beta ))p(t_{i+1}^-)(\beta ))  < 0\) for $i=0,\cdots , k(\beta )-1$ which is not possible because the degree of the polynomial function $p$.

3) It remains to choose $\lambda $ in such a way  that $\|f-\lambda \gamma \|< \| f \|$. The stronger decrease corresponds to the optimal solution 
$\lambda _{opt}$ of the minimisation problem
\begin{equation}\label{pl}
  \mu_{opt} = \min_{\lambda \geq 0}\, \max_{t\in [a,b]}[ \, |f(t)-\lambda \gamma (t)| \,] =    
  \min_{\lambda , \mu}   \,\left [\, \mu  :  \left.\begin{array}{c}f(t)-\lambda \gamma (t)\leq \mu , \\-f(t)+\lambda \gamma (t)\leq \mu,
   \\ \forall \, t\in [a,b].\end{array}\right.     \,\right ].
\end{equation}
The second formulation in problem~(\ref{pl}) is a linear programming problem with only two variables but with an infinite number of   constraints, hence   $\lambda _{opt} $ and $\mu_{opt} $  cannot be easily  obtained. In order to obtain  upper-bounds of $\mu_{opt} $,   let us introduce 
$$m_-(\beta)=\min_t \left [\,|\gamma (t) | \,:\, t \in [\,t ^-_i(\beta ), t ^+_i(\beta )\,]\;\textrm { for some }i \,\right],$$
 $$m_+(\beta)=\max_t \left[\, | \gamma (t) |\,:\, t \in [\,a,b\,\right]\ \,], \quad \rho (\beta )= m_- (\beta )(m_+(\beta))^{-1}.$$

It follows from the  construction of  the function $\gamma $ that  $0<\rho(\beta) \leq 1$ and   
$m_-(\beta )=\min_i \left [\,\gamma( t ^-_i(\beta )),\gamma ( t ^+_i(\beta ))\,\right ]$.

 Assume that  $t \in [\,t ^-_i(\beta ), t ^+_i(\beta )\,]$ for some $i$. Then $\varepsilon (\beta ) (-1)^i \gamma (t)>0$\\ and  therefore,  for all $\lambda >0$,
\begin{equation}\label{eq:pl1}
-\beta\, \|f\|-\lambda m_+(\beta) < \varepsilon (\beta)  (-1)^i  (f-\lambda \gamma) (t)  \leq \|f\|-\lambda m_-(\beta ).
\end{equation}
Next,  for  the other $t\in [a,b]$,   one has 
\begin{equation}\label{eq:pl2}
-\beta \, \|f\|-\lambda m_-(\beta ) <  (f-\lambda \gamma) (t)  <  \beta \,\|f\|+\lambda m_+(\beta).
\end{equation}
It follows immediately from the inequalities \eqref{eq:pl1}-\eqref{eq:pl2} that for \(\lambda>0\) small enough, \(\|f-\lambda\gamma\|<\|f\|\). Furthermore, for \(\beta\in (0,1)\) and $\lambda >0$
\begin{equation}\label{eq:pl3}
\|f-\lambda \gamma \| \leq \max[\, \|f\|-\lambda m_-(\beta ), \beta \, \|f\|+\lambda m_+(\beta)\,]   .
\end{equation}
In particular,  for $\bar \lambda = \|f\| \,(1-\beta )(m_-(\beta )+m_+(\beta))^{-1}$, 
\begin{equation}\label{rho}
 \mu_{opt}  \leq  \|f-\bar \lambda \gamma \| \leq  \frac{1+\beta \rho (\beta )}{1+\rho (\beta )} \,\|f \|  < \|f\|.
\end{equation}

 These bounds are  very rough:   $\bar \lambda $ is not optimal     and the  inequalities (\ref{eq:pl1}),  (\ref{eq:pl2}) and (\ref{eq:pl3}) correspond to the worst possible cases.  \end{proof}

   \vspace{0.5cm}
   
   \begin{defin}
   Consider a sequence of points
   $$a\leq t_0<t_1<\dots<t_k\leq b,~k>0.$$
   We call this sequence of points a $\beta$-alternating sequence if there exists $\varepsilon=\{1,-1\}$, such that
$$ \varepsilon f(t_0)\geq \beta M, \quad  (-1)^{i}\varepsilon f(t_{i})\geq \beta  M,~i=1,\dots,k.$$   
   \end{defin}

 The next proposition analyses the behaviour of $k(\beta )$ and the points $t_i^-(\beta )$ and $t_i^+(\beta )$ when $\beta \rightarrow 1$. 
\begin{pro}\label{k(beta)}
a) $k(\beta _1) \geq k(\beta _2)$ when $0<\beta _1 < \beta _2\leq 1$.\\ 
b) There is $\hat \beta\in (0,1)$ such that $k(\beta )=k(1)$ for all $\beta \in\,  (\hat \beta, 1]$.\\
c) For all $i$,  $t ^-_i(\beta) \rightarrow t^-_i$ and  $t ^+_i(\beta ) \rightarrow t^+_i$ when $\beta \rightarrow 1$.
\end{pro}
\begin{proof}a) It is clear that $t ^-_0(\beta _1) \leq t ^-_0(\beta _2)$. Next, $t ^-_1(\beta _1) \leq t ^-_1(\beta _2)$,  \\
$t ^-_i(\beta _1) \leq t ^-_i(\beta _2)$ for $i\geq 2$.
\vspace{0.2cm}

b) Denote by $r^-_i$ and $r^+_i$ respectively  the  smallest and the greatest \\$t\in [t^+_i,t^-_{i+1}]$ such that $f(t)=0$. Then,
$t^+_i<r^-_i \leq r^+_i<t^-_{i+1}]$.
Next, set
 $$\alpha _{-1}=\max_t\,[-\varepsilon f(t)\, :\, a\leq t \leq t^-_0\,],$$
$$\alpha _{i}=\max_t\,[ \,|f(t)| \, :\, r^-_i\leq t \leq r^+_{i}\, ], \;\, i=0,\cdots,  k(1)-1, $$
$$\alpha _{k(1)}=\max_t\,[ \,-\varepsilon (-1)^{k(1)} f(t)\, :\, t^+_{k(1)}\leq t \leq b\, ]. $$
Next, let $\alpha =\max _i[\,\alpha_i\,  ]$.  By construction, $0<\alpha <M$. Take  $\hat  \beta =\alpha M^{-1}$.
Let any  $\beta \in (\hat \beta, 1)$.
\begin{enumerate}
 \item  $ t ^-_0 (\beta )\in [a,t^-_0]$ and  $ t ^-_1(\beta )\notin [a,r^+_0]$ since   $-\beta M< \varepsilon f(t)$ for all $t\in [a,r^+_0]$ and, in case where $a\neq t^-_0$, $\varepsilon f(t^-_0)=M$.

Next, $ t ^-_1(\beta )\in (r^+_0, t^-_1)$ because  $f(r^+_0)=0$, $\varepsilon f(t^-_1)=- M$  and   \\$0> \varepsilon f(t) $ for all $t\in (r^+_0 ,t^{-1}_1)$.

Finally,  $ t ^+_0(\beta )\in (t^+_0, r^-_0)$  because  $\varepsilon f(t^+_0)=M$, $f(r^-_0)=0$  and \\ $ | f(t) | <\beta M$ for all $t\in [r^-_0,  t ^-_1)$.

 \item   
$ t ^-_2(\beta )\notin [t ^-_1(\beta ),r^+_1]$ since   $-\beta M< (-1)^1\varepsilon f(t)$ for all  $t\in [t ^-_1,r^+_1]$. 

Next, $ t ^-_2(\beta )\in (r^+_1, t^-_2)$ because  $f(r^+_1)=0$, $(-1)^2\varepsilon f(t^-_2)=M$  and  \\ $0>(-1)^1 \varepsilon f(t) $ for all $t\in (r^+_0 ,t^{-1}_1)$.

Finally,  $ t ^+_1(\beta )\in (t^+_1, r^-_2)$  because  $(-1)^1\varepsilon f(t^+_1)=M$, $f(r^-_2)=0$  and  $ | f(t) | <\beta M$ for all $t\in [r^-_1,  t ^-_2)$.
\item Proceed similarly for other $i$.

\end{enumerate}
It follows that $k(\beta)=k(1)$ for $\beta \in [\hat \beta , 1]$.

c) Assume that $\hat \beta<\beta_1<\beta_2 <1$. Then, for all $i$,
$$r^+_{i-1}<t ^-_i(\beta _1)<t ^-_i(\beta _2)<t^-_i\leq t^+_i<t ^+_i(\beta _2)<t ^+_i(\beta _2)<r^-_i.$$
Assume, for contradiction, that $t ^-_i(\beta )$ converges to some $\bar t<t^-_i$. The definition of $t^-_i$ implies $ |f(\bar t)|<M$ in contradiction with $f$ continuous and $f(t ^-_i(\beta ))=\beta M$ and $\beta \rightarrow 1$. The other convergencies are treated similarly.
\end{proof}

\vspace{0.3cm}

 Next, let us bring our attention on  the reduction rate   in (\ref{rho}) in case where we take for $\xi_i $ the middle of the interval 
 $[t_i^+(\beta ),t_{i+1}^-(\beta )]$.
 The quantity $m_+(\beta)$ is bounded from above by $[b-a]^{k(\beta)}$. Hence, roughly speaking, 
  a small value of $\rho(\beta ) $ corresponds  for some $i$ to a small value of $m_-(\beta )$ which corresponds 
to a small value of some $\xi_i(\beta )-t_i^-(\beta )$ or $t_{i+1}^+(\beta )-\xi_i(\beta )$ and/or a small value of some $\xi_{i+1}(\beta )-\xi_i(\beta )$. Let us explicit that in terms of  continuity of the function $f$.  

$f$ being  continuous on the compact set $[a,b]$ is uniformly continuous. Hence,  for  all $\delta >0$, there is $\mu >0$ such that   $|f(t)-f(s)| <\delta $ when $ |t-s|<\mu $. This motivates the introduction of the following function 
$$\mu_f (\delta ) =\sup_{s,t,\mu }\,[\,\mu >0\, :\, s,t \in [a,b] \textrm{ and } |t-s|<\mu \Longrightarrow  |f(t)-f(s)| <\delta \,],$$
$$\mu (\delta , f)= \mu_f (\delta ) =\inf_{s,t}\,[\,|t-s|\,:\,  s,t\in [a,b], \; |f(t)-f(s)| \geq \delta \,].$$
This function is in  some way an inverse  modulus of continuity of  $f$.
It is clear that
$$0<\mu_f (\delta_1)  \leq \mu_f (\delta_2 ) \;\; \textrm{ whenever } 0<\delta_1<\delta_2. $$   

 {In case where $f$ is Lipschitz, i.e., if there exists $L$ such that 
 $$\| f(t)-f(s)|\geq L|t-s|$$ for all $s,t$, one has
 $\mu_f (\delta )\leq L^{-1}\delta $.

Let us return to our problem.  By construction of the points $t_i^- (\beta ),  t_i^+ (\beta )$, one has  $|f(t_{i+1}^-(\beta ))-f(t_{i}^+(\beta ))|=2\beta \, \|f\|$ and therefore for all $i$
$$t_{i+1}^- (\beta )- t_{i}^+(\beta ) \geq \mu_f(2\beta \, \|f\|), \quad t_{i+1}^-(\beta ) - \xi_{i}(\beta )\geq  \mu_f(2\beta \, \|f\|)/2, $$
$$\xi _i(\beta )- t_{i}^-(\beta )\geq  \mu_f(2\beta \, \|f\|)/2,\quad \xi_{i+1}(\beta )-\xi_i (\beta ) \geq \mu_f(2\beta \, \|f\|),$$ 
$$(\,t ^+_i(\beta ),t ^-_{i+1}(\beta ))\supset \, (\, \xi_i(\beta )- \mu_f(2\beta \, \|f\|)/2 \, , \,\xi_i(\beta )+\mu_f(2\beta \, \|f\|)/2)).$$
 Furthermore,
$$a\leq  t ^+_0(\beta ) \leq \xi_0(\beta)- \mu_f(2\beta \, \|f\|)/2)), \quad  \xi_{k-1}+\mu_f(2\beta \, \|f\|)/2))\leq t^-_k(\beta )\leq b.   $$

Let us observe that
$$\bigcup_i \, [\,t ^-_i(\beta ),t ^+_i(\beta )\,]  \, \subset \, [\, \bigcup_i \, (\,t ^+_i(\beta ),t ^-_{i+1}(\beta ))\,] ^c \subset \cdots$$
$$\cdots  \subset \, [\, \bigcup_i \, (\, \xi_i(\beta )- \mu_f(2\beta \, \|f\|)/2 \, , \,\xi_i(\beta )+\mu_f(2\beta \, \|f\|)/2)) \,] ^c=T_f(\beta ),$$
where $c$ stands for   complementary set.

 Hence, 
\begin{equation}\label{tfalpha}
m_-(\beta)\geq \inf_t\, [\,  \prod _{i=0}^{k(\beta )-1}\, |t-\xi_i(\beta ) | \, :\, t\in [a,b] \cap T_f(\alpha )\,].
\end{equation}
 
It remains to obtain a lower bound of the product.   Let us  introduce the following function $\Gamma _k$ which does not depend on $f$ and $\beta $. 
 $$\Gamma _k(r)=\inf_{t, \xi_i}\left [ \, \prod _{i=0}^{k-1}\, | t-\xi_i |\, :\, 
 \left.\begin{array}{c}a \leq t\leq b, \\a+r\leq \xi_0\leq \xi_1\leq \cdots \leq\xi_{k-1}+r\leq b, \\   |\xi_{i+1}-\xi_i |\geq 2r, \;   | t-\xi_i |\geq r \;\;\;\forall \,i. \end{array}\right.\right] , \; r>0.$$
 By construction  $0<\Gamma_k (r_1) \leq \Gamma_k(r_2) $ whenever $0<r_1<r_2$. One has necessarily $2kr\leq b-a$. 
 
  $\Gamma_k (r)$ can be explicitly determined. Indeed, 
 \begin{itemize}
\item  $\Gamma_2(r)=r^2 $ is reached for $\xi_0=a+r,\, \xi_1=\xi_0+2r,\,  t=\frac{\xi_0+\xi_1}{2}$.
\item $\Gamma_3(r)=3r^3$ is reached for $\xi_0=a+r,\, \xi_1=\xi_0+2r,\, \xi_2=\xi_1+2r$ and $t=\frac{\xi_0+\xi_1}{2}$.
\item $\Gamma_4(r)=3^2r^4 $ reached for   $\xi_0=a+r$,  $\xi_{i+1}=\xi_i+2r$,  $ i=0,1,2$ and 
$t=\frac{\xi_1+\xi_2}{2}$.
\item $\Gamma_5(r)=3^25r^5 $  reached for   $\xi_0=a+r$,  $\xi_{i+1}=\xi_i+2r$,  $ i=0,1,2,3$ and 
$t=\frac{\xi_1+\xi_2}{2}$.
\end{itemize}
More generally, $\Gamma_k (r)=c_kr^k$ where
$$c_{2q}=1^23^25^5\cdots (2q-1)^2, \quad c_{2q+1}=\frac{1^23^25^5\cdots (2q+1)^2}{2q+1}.$$
Since the logarithmic function is increasing
\begin{gather*}
  \int_1^{2q-1}\ln(t)dt \leq 2[\ln(3)+\ln(5)\cdots +\ln(2q-1)]; \\
 (2q-1)\ln(\frac{2q-1}{e}) +1 \leq \ln[\,3^25^2\cdots (2q-1)^2\,].
 \end{gather*}
 It follows that 
 \begin{gather*}
 \Gamma_{2q}(r)\geq \frac{e^2}{k-1}\, \left[\frac{(k-1)r}{e}\right ]^k\geq  \frac{e}{k}\,\left [\frac{(k-1)r}{e}\right]^k \;\;\textrm{ if } k =2q;\\
 \Gamma_{2q+1}(r)\geq \frac{e}{k}\, \left[\frac{kr}{e}\right]^k\geq  \frac{e}{k}\, \left[\frac{(k-1)r}{e}\right]^k \;\; \textrm{ if } k =2q+1.
  \end{gather*}
  Going back to  (\ref{tfalpha}) we obtain 
$$m_-(\beta )\geq \Gamma_{k(\beta )} (\frac{\mu_f(2\beta \, \|f\|)}{2})\geq 
 \frac{e}{k(\beta) }\, \left[\frac{(k(\beta )-1)\mu_f(2\beta \, \|f\|)}{2e}\right]^{k(\beta )}.
$$ 
Finally,
$$\rho(\beta )=\frac{m_-(\beta)}{m_+(\beta) }\geq  \frac{e}{k(\beta) }\, \left[\frac{(k(\beta )-1)\mu_f(2\beta \, \|f\|)}{2e(b-a)}\right]^{k(\beta )}.$$ 
 
 Next, since the function $\rho \rightarrow (1+\beta \rho) (1+ \rho) ^{-1}$ decreases on $[0,\infty )$, we obtain
 the following theorem which provides  an upper bound  of the reduction rate  in terms of  the degree of continuity of $f$ 
 and the parameter $\beta$.
 \begin{thm} \label{rate}In case where $\xi_i=\frac{1}{2}(t_i^+(\beta )+t_{i+1}^-(\beta ))$ for all $i$, for a suitably chosen~$\lambda$.
      \[
\|f- \lambda \gamma (\beta)\| \leq  \left[1-  \frac{(1-\beta)\tau  }{1+\tau } \right] \|f\|,\]
{ where }
$\tau = \frac{e}{k(\beta) }\, \left[\frac{(k(\beta )-1)\mu_f(2\beta \, \|f\|)}{2e(b-a)}\right]^{k(\beta )}.$
\end{thm}
This theorem will be the clue for the convergence of algorithms in the next sections.

\section{The Chebyshev alternance theorem}\label{sec:three}
 Let us denote by $\Pi_n$  the set of polynomial functions with degree less than or equal to $n$. 

The problem consists to solve the convex optimization problem
\begin{equation}\label{Tch}
 \min_ {\pi \in \Pi_n} \,\left [\,\| \pi -f \| \,\right].
\end{equation}

 Based on Theorem \ref{thm:alternance}, we present a very short and  original proof  of the celebrated result of Chebyshev on  polynomial  approximation. The intrigant $n+2$ alternate  points condition appears as the conjonction  of the alternance
 propriety on continuous functions with  the dimension of the linear space $\Pi_n$.
 
\begin{thm}[Chebyshev  theorem]\label{thm:tch} Assume that $f:[a,b]\rightarrow \R$ is continuous. Then (\ref{Tch}) has one and only one optimal solution. Furthermore, $\pi \in \Pi_n$
is the optimal solution if and only if there exist $\varepsilon \in \{-1,1\}$, $k\geq n+1$ and $t_0, t_1,\cdots , t_{k}$ such that
$$a\leq t_0<t_1<\cdots <t_n<t_{k}\leq b $$ 
and  $ f(t_i)-\pi(t_i)=\varepsilon \,  (-1)^i \, \|f-\pi\|$ for all~$i$.
\end{thm}
\begin{proof}
    \begin{enumerate}
        \item \textbf{Existence} The function $\gamma  $ defined by $\gamma  (\pi)=\| \pi -f\|$ is convex and continuous on $\Pi_n$. To prove the existence of one optimal solution it is enough to prove that the set $A:=\{\pi \in \Pi_n\,:\, \|\pi-f\|\leq \| f\|\}$ is bounded.
Given $n+1$ arbitrary distinct points $t_1,t_2, \cdots, t_{n+1}\in [a,b]$, 
$\pi \in \Pi_n$ is uniquely defined by the data of $n+1$ values  $\alpha_i$ via  the formula
$$\pi(t)= \sum_{i=1}^{n+1}\alpha_i \prod_{j\neq i}\frac{t-t_j}{t_i-t_j}.$$
If $\pi \in A$, then for all~$i$, 
$$|\alpha _i -f(t_i)| =|\pi(t_i)-f(t_i)|\leq  \|\pi-f\|\leq \| f\|.$$
Hence, $|\alpha _i |\leq  2\, \| f\|$ for all $i$. Thus, $A$ is bounded.

\item \textbf{Necessity}  Let 
$$M= \max \,[\, \pi(t)- f(t): a \leq t \leq b\, ], \quad m=\min\, [\,\pi(t)-f(t): a \leq t \leq b\,].$$

 Let us consider the case where  $M+m\neq 0$.  Set  $\delta (t)=\frac{1}{2}(M+m)$ for all~$t$. Then, $\pi-\delta \in \Pi_n$ and $\| \pi  -\delta -f\| \leq \|\pi -f \| -\frac{1}{2}(M-m)$. Hence, $\pi $ is not an optimal solution. 
 
We are left with the case $M+m=0$. Apply Theorem \ref{thm:alternance} to the function $\pi-f$. Take $t_i=t^-_i$ for $i=0,1, \cdots, k$. Assume that $k<n+1$.  
Consider the function  $\gamma  $ defined in part 2 of the theorem. Then, $\gamma  \in \Pi_n$. We have seen that,  for $\lambda >0$ small enough,
$\|\pi -\lambda \gamma  -f\| <\|\pi-f\|   $ and therefore $\pi $ is not optimal.

\item \textbf{Sufficiency and uniqueness} Next, assume that  $k\geq n+1$. Let $\hat \pi \in \Pi_n$ be  an  optimal solution. We have seen that such an optimal solution exists.  We must prove that $\pi =\hat \pi$. We have
$$\varepsilon (-1)^i [\hat \pi (t_i)-f(t_i)]\leq \|\pi -f\|=\varepsilon (-1)^i[ \pi (t_i)-f(t_i)],\quad i=0,1, \cdots, k.$$
 Hence,
$\varepsilon (-1)^i [\hat \pi -\pi](t_i) \leq 0$ for $i=0,1, \cdots, k$. Since $\hat \pi -\pi \in \Pi_n$ and $k\geq n+1$, this is possible  only if $\pi =\hat \pi$.

    \end{enumerate}
\end{proof}

%
%
%

Theorem~\ref{thm:tch} and its proof are constructive in the sense they allow to determine if some candidate  $\pi\in \Pi_n$  to optimality is optimal and in the opposite case to give a better candidate, but they cannot directly used for designing  algorithms building sequences   converging  to the optimal solution.  
Recall that that  the problem consists in the minimisation of  $\|\sigma -f\|$ subject to $\sigma \in \Pi_n$, $f$ being a fixed continuous function on $[a,b]$. Because it is numerically improbable that an arbitrary  function reaches exactly  its absolute maximum at more than two or three points,  we  shall construct algorithms  converging to a solution satisfying an approximate optimality condition.

Given $\beta \in (0,1)$, let 
$$M(\sigma )=\max_{t\in [a,b]}\,[\sigma (t)-f(t)],  \quad m(\sigma )=\min_{t\in [a,b]}\,[\sigma (t)-f(t)], $$
$$\hat \sigma =\sigma-\frac{1}{2}[M(\sigma )+m(\sigma )], \quad \theta (\sigma)= \|\hat \sigma-f\|=\frac{1}{2}[M(\sigma )-m(\sigma )].$$
The functions $M, -m$ and $\theta $ are convex and defined on the whole space $\Pi_n$.

Next, given $\sigma \in \Pi_n$, let us denote by $k(\beta, \sigma )$ the integer corresponding to  $k(\beta)$ in Theorem \ref{thm:num}  applied to $ \hat \sigma-f$ in  place of $f$.  
 
\begin{pro}\label{pro:neigh}Let $\pi \in \Pi_n$ and $\beta , \beta'$ such that $0<\beta' <\beta <1$. There exists a neighbourhood $V$ of $\pi $ in $\Pi_n$ such that  $k(\beta', \sigma )\geq k(\beta, \pi)$  for all $\sigma \in V$.
\end{pro}
\begin{proof}There exist $\xi_i\in [a,b]$ and $\varepsilon \in \{-1,1\}$ such that
$$a\leq \xi_0<\xi_1<\cdots <\xi_{k(\beta,\pi)} \leq b \;\; \textrm { and } $$
$$\varepsilon (-1)^{i}( \pi-f)(\xi_i)\geq \beta M(\pi) >\beta' M( \pi) \quad \forall \,i.$$
Due to the continuity of the functions $M$ and $m$, there is  a neighbourhood $V$ of $\pi $ in $\Pi_n$  such that for all $ \sigma \in V$
$$\varepsilon (-1)^{i}(\sigma-f)(\xi_i) >\beta' M  (\sigma )\quad \forall \,i=0,1, \cdots,  k(\beta, \pi).$$
It follows  $k(\beta', \sigma )\geq k(\beta, \pi)$. \end{proof}

 Given $\sigma \in \Pi_n$ and $\beta \in (0,1)$, 
we say that $\sigma \in \Pi_n$  fulfills  the \textbf{$\beta$-alternance  optimality condition} for the problem (\ref{Tch}) if 	$k(\beta, \sigma )\geq n+1$.

\begin{pro} \label{pro:neigh2}  a) If $\bar \pi \in \Pi_n$ is the optimal solution to problem (\ref{Tch})   there is $\bar \beta \in (0,1)$ such that 
$k(\beta ', \bar \pi )\geq n+1$ for all $\beta \in [\bar \beta ,1]$. Furthermore,
for any $\beta \in (\bar \beta ,1]$, there exists  a neighbourhood $V_\beta $ of $\bar \pi$ such that $k(\beta ',  \sigma )\geq n+1$ 
for all $\sigma  \in V_ \beta $ and $\beta '\leq \beta$.

 b) If $\bar \pi \in \Pi_n$ is  not the optimal solution of problem~(\ref{Tch}),   there is $ \beta \in (0,1)$ such that 
$k(\beta ', \bar  \pi )< n+1$ for all  $\beta '\ \in [\beta, 1]$. 
 \end{pro}
\begin{proof}The proposition is a consequence of Theorem \ref{thm:tch},  Proposition~\ref{k(beta)} and Proposition~\ref{pro:neigh}. \end{proof}

Based on Proposition~\ref{pro:neigh2}, we introduce the following approximate optimality  condition.  Given $\sigma \in \Pi_n$ and $  \beta \in (0,1)$, 
we say that $\sigma \in \Pi_n$  fulfills  the \textbf{$\beta$-alternance  optimality condition} for the problem (\ref{Tch}) if 	$k( \beta, \sigma )\geq n+1$.

Now, we are ready to propose an  algorithm converging to some $\sigma \in \Pi_n$  fulfilling  this optimality condition. It supposes that we have at our disposition an auxiliary algorithm giving a rather good estimation of the maximum of a continuous function on the closed interval $[a,b]$.

\SetNlSty{textbf}{}{}
\begin{algorithm}[H]
    \caption{2 knots}\label{alg:2knots}
         \DontPrintSemicolon
    
         \KwIn{Together with $\beta^+ ,\beta^- $  with  $0<\beta^- <\beta^+<1$  two fixed  parameters $\gamma^-, \gamma ^+$  such that $0<\gamma^ -\leq 1<\gamma ^+$ are given.}
\Init{
\nl  Start with $\sigma $  defined by $\sigma (t)=\frac{t-a}{b-a}f(b)+\frac{t-b}{a-b}f(a).$\;
\nl  Start with $\beta \in [\beta^-,\beta^+ \, ]$.\;}
\Main{
    \setcounter{AlgoLine}{0}
\nl Compute $M=\max_{t\in [a,b]}\,[\sigma (t)-f(t)]$,
$m=\min_{t\in [a,b]}\,[\sigma (t)-f(t)]$.\;
\nl  Update $\sigma $:  take  $\sigma (t)= \sigma (t)-\frac{M+m}{2}$ for all $t$.\;
\nl \label{step:thmnum} Apply Theorem \ref{thm:num} to $\sigma -f$.\; 
\nl If $k(\beta )\geq n+1$ and $\beta = \beta^+ : $\textbf{ STOP,}  we have found a $\beta^+ $-approximation of the   solution.\;
\nl If $k(\beta )\geq n+1$ and $\beta < \beta^+ $: take  $\beta = \min\,[\, \gamma^+\beta, \beta^+\,]$ and return to~\ref{step:thmnum}.\;
\nl  If $k(\beta )\leq n$:  set   $\xi_i =\frac{t ^+_i(\beta)+t ^-_{i+1}(\beta)}{2}$,  $i=0,1, \cdots , k(\beta )-1.$\\
  Set $\gamma  (t)=  \varepsilon (\beta ) \prod _{i=0}^{k(\beta )-1}\, (\xi_i-t)$.\\
 Choose $\lambda >0$ such that $\| \sigma -f -\lambda \gamma \| < \frac{1+\beta \rho (\beta )}{1+\rho (\beta )} \,\|\sigma -f\|$,\\
such $\lambda $  exist in view of Theorem \ref{thm:num}. \\
  Do $\sigma =\sigma -\lambda \gamma $ and $\beta = \max\, [\,\beta^-, \gamma^-\beta \,] $.  \\
  Go to main step.}
\end{algorithm}
 
\begin{thm}
    The algorithm converges in a finite number of steps to some $\hat \sigma \in \Pi_n$ such that $k( \beta  ^+,\hat\sigma )\geq n+1$. Furthermore, \(\beta ^+\,\| \hat \sigma - f\| \leq \|  \pi_n -f \|\).
\end{thm}
\begin{proof} 
    The second part of the theorem follows from the first, and part 2) of Theorem~\ref{thm:num}.
    Assume  for contradiction that the algorithm does not stop in a finite number of iterations. Let us denote by $\pi _n$ the unique optimal solution of problem (\ref{Tch}) and by $\sigma ^l$ the polynomial function at the $l$-th iteration after substep 6. of main step.  By construction,
 $$0\leq  \| \pi_n -f  \|  \leq   \| f-\sigma ^{l+1}  \| <  \|f- \sigma ^l \| \quad \forall \, l.$$
 Set $\alpha = \lim_l  \|f- \sigma ^l \|$.

  The function $\sigma \rightarrow \| \sigma -f\| $ is convex and continuous, it reaches its minimum on $\Pi_n$ at one unique point. It follows the compactness of the set 
 $$\Sigma =\{ \,\sigma \in \Pi_n \, :\, \|f-\sigma  \|\leq \|f-\sigma ^0 \| \, \}.$$
 Let us define for $\delta >0$
 $$\bar \mu (\delta )=\inf_{\sigma}\,[\,\mu (\delta ,  f-\sigma )\,:\,\sigma \in \Sigma \,].$$
 More explicitly,
$$\bar \mu (\delta )=\inf_{t,s,\sigma}\,[\,|t-s]: s,t\in [a,b], \sigma \in \Sigma, \; |f(t)-\sigma (t)-f(s)+\sigma (s)| \geq \delta \,].$$
Due to the compactness of the sets $[a,b]$ and $\Sigma $, the infimum is reached at some $(\bar s,\bar t)\in [a,b]^2$ and $\bar \sigma \in \Sigma $. Hence,  $\bar \mu (\delta) >0$.

 i) Firstly, consider the case where  $\alpha >0$. Set 
$$\bar \tau =\min_{k\leq n} \, \left [\frac{1}{k}\,\left (\,\frac{k\mu_f(2 \beta ^- \alpha )}{2e(b-a)}\right )^k\,\right]  .$$
Then $\bar \tau >0$. Since $\beta ^- \alpha \leq \beta ^l \|f-\sigma ^l\|$, Theorem \ref{rate} implies
$$ 0<\alpha <\| f-\sigma ^{l+1}  \|  \leq  \left[\,1-  \frac{(1-\beta^+)\bar \tau  }{1+\bar \tau }\, \right]\, \ \| f-\sigma ^{l}  \| \quad \forall \, l $$
which is not possible.

  ii) It remains to consider the case where $\alpha =0$. Then  $f=\pi_n$ and  the whole sequence $\{\sigma ^l\}$  converges to the optimal solution $\bar \pi= \pi_n$  of  problem (\ref{Tch}). 
 
 Let $\bar \beta $ as in a) of  Proposition \ref{pro:neigh2}. Set $\hat \beta =\max\,[\, \hat \beta, \beta ^+\,]$ and let 
$V$ neighbourhood  of $ \pi_n $ such that $k(\beta', \sigma )\geq n+1$ for all $ \sigma \in V$ and $\beta'\leq \hat \beta $. Since   
$\sigma ^l$  goes  to $\pi_n$ when $l$ goes to $+\infty$, there is a finite integer $ l_0$ such that $\sigma^{l_0}\in V$ and thereby   $k(\beta^{l_0},   \sigma ^{l^0})\geq n+1$. After a finite number of iterations of item 5   where $\sigma $ remains unchanged,
we are in situation 4. The algorithm stops.  \end{proof}

 \section{Spline approximation with $p+2$ fixed knots}\label{sec:four}
In this section, we are given  $p+1$ integers $n_i\geq 1$, $i=0,\cdots , p$   and $p+2$ points $x_i\in [a,b]$ such that 
 $$a=x_0<x_1<x_2< \cdots <x_{p}<x_{p+1} = b .$$
  
We define  $\Sigma $ as the set of functions $\sigma $ on $[a,b]$  such that  for $i=0, \cdots ,p$ there exist $\sigma _i\in \Pi_{n_i}$ such that
$$\sigma(t)=\sigma_i(t)\;\;\;\forall \,t \in I_i:=[x_{i},x_{i+1}] ,\quad \sigma_i(x_{i+1})=\sigma_{i+1}(x_{i+1}),~i=0,\dots,p.$$

The functions $\sigma \in \Sigma $ are called splines, the points $x_i$ are called knots. $\Sigma $ is a linear space with dimension $n_0+n_1+\cdots +n_p$.  We use the following notation $\sigma=(\sigma_0,\sigma_1,\cdots , \sigma_p)$.

We are concerned with the convex optimization problem
\begin{equation}\label{fixed}
\min_{\sigma \in \Sigma  } \| \sigma -f \|,
\end{equation}
where $f$ is continuous on $[a,b]$.

Using a similar argument as in section 2,  it is not difficult to see that the set $A=\{\sigma \in \Sigma \,:\, \| \sigma -f \| \leq \|f\| \, \}$ is bounded. Hence, the problem has at least one optimal solution.

Set $\| \sigma -f \|_k=\max_{t\in I_k}| \sigma (t)-f(t)| =\max_{t\in I_k}| \sigma _k(t) -f(t) |$. Without loss of coherence, we write 
$\| \sigma  -f \|_k=\|\sigma _k -f \|_k$. Then, $\|\sigma -f \| =\max_k\| \sigma -f \|_k$.

\vspace{0.2cm}
Let $\sigma \in \Sigma $. Set
$$M:=\max_{t\in [a,b]}[\sigma (t)-f(t)], \quad   m:=\min_{t\in [a,b]}[\sigma (t)-f(t)].$$ 

It is clear that $M+m=0$ is a necessary condition for $\sigma $ to be an optimal spline.   To see that, take $\delta (t)=\frac{1}{2}(M+m)$ for all $t$. Then,   $\sigma -\delta\in \Sigma $ and 
 $\| \sigma -\delta -f\| \leq \|\sigma -f \| -\frac{1}{2}(M-m)$.

\subsection{A sufficient  condition for optimality} 

 Assume that $M+m=0$ and $\sigma \in \Sigma $ is not optimal. Then,   there is some $\delta \in \Sigma $ such that 
$ \|\sigma +\delta -f \|< \|\sigma -f \|$.   Then $\varepsilon (-1)^j \delta (t^-_j)<0$  and $\varepsilon (-1)^j \delta (t^+_j)<0$,  for all $ j=0,1, \cdots , k$ where $k$ and the points $t^-_j$ and $t^+_j$ are obtained via Theorem \ref{thm:alternance} applied  to the continuous function $\sigma  -f$.

 Without loss of generality, $ \delta $ can be taken so that $\delta (x_i)\neq 0$ for all $i$.  If not, add to $\delta $ a constant function    $\iota $ such that  $(\delta+\iota)(x_i) \neq 0$ for all $i$ and small enough to have  $ \|\sigma +\delta + \iota -f \|< \|\sigma -f \|$. 
  
  \vspace{0.2cm}
  
 The number of roots of the equation $\delta (t)=0$ in the interval  $ [t^+_{j-1},t^-_{j}]$ is a strictly positive odd number.  The number of roots contained in the interval $ [t^-_{j},t^+_{j}]$ is even, possibly $0$. More generally, the number of roots  of the equation   $\delta (t )=0$  in the interval $[t^-_{j},t^+_{j+l}]$  is of the form $l+2m$ with $m\geq 0$  integer.   
 
 Given $i_1,i_2$ with $0\leq i_1< i_2\leq p$, set 
 $$J(i_1,i_2)=\{\,j\,:\,t^+_j \textrm { and/or  }t^-_{j}\in [x_{i_1},x_{i_2}\,]\}.$$
  It follows that the number of roots  of the equation $\delta (t)=0$  in the interval $ [x_{i_1},x_{i_2}]$ is at least $\card (J(i_1,i_2))-1$.
  
Recall that $\delta (x_i)\neq 0$ for all $i$  and, on each interval  $ [x_{i},x_{i+1}]$, $\delta $ is a polynomial function $\delta_i$ with degree less or equal to $n_i$. Therefore, the total number of roots of the equation $\delta (t)=0$  in the interval $ [x_{i_1},x_{i_2}]$ is at most~$n_{i_1}+n_{i_1+1}+ \cdots +n_{i_2-1}$. Hence, for the existence of $\delta$, it is necessary that for all $i_1,i_2$ such that $i_1<i_2$  the following holds 
$$\card (J(i_1,i_2))-1\leq n_{i_1}+n_{i_1+1}+ \cdots +n_{i_2-1}.$$

We have proved the following two propositions.
\begin{pro}\label{suff}Assume $M+m=0$.   
A sufficient condition for $\sigma  $ to be  optimal  is  
$$ \exists  \, i_1,i_2, \; i_1 <  i_2 ,\; \textrm{ such that } \; \card(J(i_1,i_2))\geq n_{i_1}+ n_{i_1+1}+ \cdots +n_{i_2-1}  +2.\eqno{(\textbf{CS})}$$
\end{pro}

\begin{pro}\label{suff2}  
Assume that we are given  $q$ points $\xi_i$ such that\\
$a\leq \xi_1 <\xi_2< \cdots <\xi_{q-1}<\xi_q \leq b$. If $q \geq n_0+ n_{1}+ \cdots +n_{p}  +2$, there is no $\delta \in \Sigma $ 
such that $\delta (\xi_i) \delta (\xi _{i+1})<0$ for $i=1, \cdots,  q-1$.
\end{pro}

\vspace{0.2cm}

\subsection{A sufficient  condition for nonoptimality} 

\begin{pro} \label{gamma}
Assume  $M+m=0$. Let the integer $k$, $\varepsilon $ and the points $t^-_j$ and $t^+_j$  obtained from  Theorem \ref{thm:alternance} applied  to the  function $\sigma  -f$. Assume that  for each $j$  there is      $ \xi_j\in (t^+_{j},t^-_{j+1})$  such that  $\xi_{j}\neq x_i$ for all $i,j$   and the number  $r_i$    of $\xi_j$ belonging to  the interval $[x_{i},x_{i+1}]$ is less than or equal to $n_i$. Then $\sigma $ is not an optimal spline.
\end{pro}
\begin{proof} We build a function $\gamma :[a,b]\rightarrow \R$ and  functions $\gamma_i : [x_i,x_{i+1}]\rightarrow \R$ as follows:
\begin{equation}\label{gammafromxi}
\gamma (t)=\varepsilon \prod_{j=0}^{k-1}\,(\xi_j-t ),\quad \gamma_i(t)=\varepsilon _i\prod_{j\in J(i)}\,(\xi_j-t), \;\; i=0, \cdots, p,
\end{equation}
where $J(i)=\{j \in (x_i,x_{i+1}) \}$,    and, for each $i$,   $\varepsilon _i\in \{-1,1\}$  is chosen so that the signs of $\gamma $ and $\gamma _i$ are the  same  on the interval $I_i=[x_i,x_{i+1}]$.  In case where  $r_i=\card(J(i))=0$, the product is taken equal to $1$.  

Set $\delta_0=-\gamma_0$ and, for $i=1, \cdots ,p$, $\delta_i=\alpha _i \gamma _i$ where $\alpha_i >0$ is taken so that $\delta _i(x_i)=\delta_{i-1}(x_i)$. By construction, $\delta =(\delta_0, \delta _1, \cdots , \delta_p) \in \Sigma$ and 
$$\varepsilon (-1)^{j}\delta (t)>0 \textrm{ if }\;\; \xi_j<t< \xi_{j+1},\; \forall \, j=0,1,\cdots , k,$$
and, in particular,$$\varepsilon (-1)^{j}\delta (t)<0 \textrm{ if }\;\; t^-_{j}\leq t\leq  t^+_{j},\; \forall \, j=0,1,\cdots , k.$$

We shall prove that,  for $\lambda >0$ small enough,  
\begin{equation}\label{inq}
| \sigma (t)+\lambda \delta (t)-f(t)|<M  \quad \forall  \, t\in [a,b].
\end{equation}
Assume,  for contradiction, that for each positive positive integer $m$ there is some $t_m\in [a,b]$ such that 
\begin{equation}\label{contradiction}
| \sigma (t_m)+\frac{1}{m}\delta (t_m)-f(t_m)|\geq M   .
\end{equation}
Let $\bar t$ be a cluster point of the sequence $\{t_m\}$. Then, $| \sigma (\bar t)-f(\bar t)| =  M$ and thereby there is some $j$
such that $\bar t\in [t^-_j,t^+_j]$. For $t$ in  a neighbourhood of $\bar t$ one has $-M<\varepsilon (-1)^j (\sigma (t)-f(t))\leq M$ and 
$\varepsilon  (-1)^{j}\delta (t)<0$. For $m$ large enough, $t_m$ belongs to the neighbourhood  and therefore
$$-M<\varepsilon (-1)^j (\sigma (t_m)+\frac{1}{m}\delta (t_m)-f(t_m))< M,$$ 
in contradiction with (\ref{contradiction}).

Since (\ref{inq}) holds for small $\lambda $, $\sigma $ is not an optimal spline. 
\end{proof}

\vspace{0.2cm}

\subsection{Condition (CS) is  necessary and sufficient   for optimality} 


Now,  we present an algorithm  which,  in case where  condition (CS) does not hold,  builds   points $\xi_j$ which  fulfill the conditions of Proposition
\ref{gamma}.

  Firstly, observe that, in case where  (CS) does not hold, necessarily for all $i$ one has $\card(J(i,i+1))\leq n_i+1$ and, in case where $\card(J(i,i+1))=n_i+1$,  $\card(J(i-1,i))\leq n_{i-1}$  and $\card(J(i+1,i+2))\leq n_{i+1}$.

  \begin{algorithm}[H]
      \caption{Construction  of intermediary points $\xi_j$} \label{alg:itermediary}
      \nlset{Initialisation} Set $r_i=0$ and $J(i)=\emptyset $ for all $i$. \\Set $j=0$.\;
\nlset{Step j}
Let $i$ be  such that $t^+_j\in [x_i,x_{i+1})$.We consider four cases:
\begin{description}
\item[First case:]$x_i\leq t^+_j<t^-_{j+1} \leq x_{i+1}$  and   $r_i<n_i$.\\
 Do  $\xi _{j}=\frac{1}{2}[t^+_j+t^-_{j+1}]$. By construction  $\xi_{j}\in (x_i,x_{i+1})$.\\
Do $r_i=r_i +1$ and  $J(i)=J(i)\cup \{j\}$.\\
If $j=k$ go to End, otherwise do $j=j+1$ and  go to Step $j$.

  \item [Second case:] $x_i\leq t^+_j <x_{i+1}<t^-_{j+1}$ and  $r_i<n_i$.\\  
  Do $\xi _{j}=\frac{1}{2}[t^+_j+x_{i+1}]$. By construction, $\xi_{j}\in (x_i,x_{i+1})$.\\
Do $r_i=r_i +1$ and  $J(i)=J(i)\cup \{j\}$. \\
If $j=k$ go to End, otherwise do $j=j+1$ and  go to Step $j$.
 
\item[Third case:] $x_i\leq t^+_j <x_{i+1}<t^-_{j+1}$ and  $r_i=n_i$.\\  
  Do  $\xi _{j}=\frac{1}{2}[ x_{i+1}+\min(x_{i+2},t^-_{j+1})\, ]$. By construction, $\xi_j \in (x_{i+1},x_{i+2})$.\\
Do $i=i+1$, $r_{i}=1$ and $J(i)=\{\xi_j\}$.\\
If $j=k$ go to End, otherwise do $j=j+1$ and  go to Step $j$.
 
\item[Fourth case:] $x_i\leq t^+_j<t^-_{j+1} \leq x_{i+1}$ and  $r_i=n_i$.\\  \textbf{Stop: (CS)  holds, $\sigma $ is one optimal spline}.
\end{description}\;
\nlset{End}

$\xi_0, \xi_1, \cdots , \xi_k $ fulfilling the requirements of Proposition \ref{gamma} have been obtained. Hence,  \textbf{$\sigma $ is not one optimal spline.}
 
  \end{algorithm}

\begin{proof} It is enough to  prove  that  condition $(CS)$ holds as soon as we encounter the fourth case. We are faced with 
  $i$ and $j$ such that $r_l \leq  n_l$ for all $l< i$,  $r_i=n_i$ and $x_i\leq t^+_j<t^-_{j+1} \leq x_{i+1}$.  Then,
$$x_i<\xi_{j-n_i}<t^-_{j+1-n_i}\leq t^+_{j+1-n_i}<\xi_{j+1-n_i}<\cdots  < \xi_{j-1} <t^-_j\leq t^+_j<t^-_{j+1} \leq x_{i+1}.$$
There are two possibilities:

i)  $x_i\leq t^+_{j-n_i} <\xi _{j-n_i}< t^-_{j+1-n_i}<\cdots  < t^+_j < t^-_{j+1}\leq x_{i+1}$: \\Then $\card (J(i,i+1) )\geq n_i+2$. Hence, condition (CS) holds. 

ii)  $t^+_{j-n_i}<x_i < \xi_{j-n_i}< t^-_{j+1-n_i}< \cdots  < t^+_j < t^-_{j+1}\leq x_{i+1}$:\\
Report to the determination  of $\xi_{j-n_i}$  according to the rules given in the second and third  cases.  Necessarily,  $n_{i-1}=r_{i-1}$.
For simplicity, set $l=j-n_i$, one has
$$x_{i-1} <\xi_{l-n_{i-1}}<t^-_{j+1-n_i}<\cdots  < \xi_{l-1} <t^-_{l} \leq x_{i}.$$
There are two possibilities:

a)  $x_{i-1}\leq t^+_{l-n_{i-1}} <\xi _{l-n_{i-1}}< t^-_{j+1-n_i}<\cdots  <t^-_{l}\leq x_{i}$: \\Then $\card (J(i-1,i+1) )\geq n_{i-1}+n_i+2$. Hence, condition (CS) holds.

 b)  $t^+_{l-n_{i-1}}<x_{i-1} < \xi_{l-n_{i-1}} < \cdots  <  t^-_{l}\leq x_{i}$:\\
 Report to the obtention of $\xi_{l-n_{i-1}}$. Necessarily, $n_{i-2}=r_{i-2}$.
For simplicity, set $m=l-n_{i-1}$, one has
$$x_{i-2} <\xi_{m-n_{i-2}}<\cdots   < \xi_{m-1} <t^-_{m} \leq x_{i-1}.$$
There are two possibilities, etc,  repeat the process as long as necessary.\end{proof}

The following theorem is a consequence of the construction.

\begin{thm}\label{thm:spl} Assume that $f:[a,b]\rightarrow \R$ is continuous and  $\sigma \in \Sigma $ is such that 
$M:= \max \,[\, \sigma (t)-f(t): a \leq t \leq b\, ]=- \min\, [\,\sigma (t)-f(t): a \leq t \leq b\,]>0$.  Then   $\sigma \in \Sigma $ is an optimal solution of   (\ref{fixed})
if and only if (CS) holds. \end{thm}

\vspace{0.2cm}

%


We are ready for describing a prototype algorithm. As in Algorithm~\ref{algo:construction}, we seek splines satisfying a $\beta$-optimal condition with $\beta \in (0,1)$  close to $1$. 

\begin{algorithm}[H]
     \caption{$p+2$ fixed knots}\label{algo:pplus2}
\Init{
 Start with $\sigma =(\sigma_0,\sigma_1,\cdots, \sigma_{p})$  defined by
 $$\sigma _i(t)=\frac{t-x_i}{x_{i+1}-x_i}f(x_{i+1})+
\frac{t-x_{i+1}}{x_i-x_{i+1}}f(x_i), \quad i=0,1, \cdots , p.$$\;}
\Main{
\nl Find  $M=\max_{t\in [a,b]}[\sigma (t)-f(t)]$, 
$m=\min_{t\in [a,b]}[\sigma (t)-f(t)]$.\;
\nl  Update $\sigma $:  take  $\sigma (t)= \sigma (t)-(M+m)/2$ for all $t$.\;
\nl Apply Theorem \ref{thm:num} to $\sigma -f$. One obtains the quantities $\varepsilon (\beta),\, k(\beta),\, t^-_j(\beta),\,  t^+_j(\beta)$.\;
\nl Apply  Algorithm 2 with  $\varepsilon (\beta),\, k(\beta),\, t^-_j(\beta),\,  t^+_j(\beta)$ in place of 
$\varepsilon,\, k,\, t^-_j,\,  t^+_j$ to obtain intermediary values
 $\xi_j\in ( t^+_j(\beta), t^-_{j+1}(\beta))$.\;
\nl If Algorithm 2 stops at the fourth case of  step $j$ with $k(\beta )<j$:  \textbf{ STOP,}  we have found a  $\beta $-optimal spline.\;
\nl If Algorithm 2 stops at  step $j$ with $k(\beta )=j$, construct functions, $\gamma $, $\gamma _i$, $\delta _i$ and $\delta $ as in 
Proposition \ref{gamma}. Next choose $\lambda >0$ such that $\| \sigma +\lambda \gamma -f\| <   \| \sigma -f\|$. \;
\nl Do $\sigma =\sigma +\lambda \gamma $ and go to main step.\;
}
\end{algorithm}

\section{The free knots spline approximation problem}\label{sec:five}
\subsection{Local properties of the objective function}
In this section, we are given  $p+1$ integers $n_i\geq 1$, $i=0,1,\cdots , p$   and 
$$X=\{ x\in  [a,b]^{n+2} \,:\, a=x_0<x_1<x_2< \cdots <x_{p}<x_{p+1} = b \}.$$

Given $x\in X$,   $\Sigma (x) $ is the set of functions $\sigma $ on $[a,b]$  such that  for each $i=0,1, \cdots ,p-1$ there exists $\sigma _i\in \Pi_{n_i}$ such that
$$\sigma(t)=\sigma_i(t)\;\;\;\forall \,t \in I_i(x):=[x_i,x_{i+1}] ,\quad i=0,\dots,p$$
and 
$$\sigma_i(x_{i+1})=\sigma_{i+1}(x_{i+1}),~i=0,\dots,p-1.$$
Here again, we use the following notation $\sigma=(\sigma_0,\sigma_1,\cdots , \sigma_p)$.

The free knots problem is the minimisation  problem 
\begin{equation}\label{free}
 \min_{x,\sigma } \,[ \,\| \sigma -f \| \,:\, x\in X, \, \sigma \in \Sigma ( x)\,], 
  \end{equation}
  where $f$ is a given continuous function  on $[a,b]$.  Set 
  \begin{equation}\label{theta}
\theta (x) =\min_{\sigma \in \Sigma ( x) } \|  \sigma -f \|. 
 \end{equation}
Then $\theta (x)\leq \|f\|$  for any $x$ since the null function belongs to $\Sigma (x)$. The next proposition is concerned with  the continuity properties of the function $\theta$.

\begin{pro}$\theta $ is locally Lipschitz on $ X$.
\end{pro}

\begin{proof}i) Let $\bar x \in X$. Set for $i=0,1, \cdots, p$ and $k=1,2,\cdots, n_i+1$
$$t_i^k=\bar x_i+\frac{k}{n_i+2}\, (\bar x_{i+1}-\bar x_i).$$
We consider a neighbourhood $V$ of $\bar x$ such that  for all $x\in V$
$$ x_i<t^1_i<t_i^{n_i+1}<x_{i+1} \quad  i=0,1,\cdots , p.$$
Let $x\in V$ and $\sigma  =(\sigma_0, \sigma _1, \cdots , \sigma _p) \in \Sigma(x)  $ be an optimal solution  of  (\ref{theta}). We know by Section 4 that such a $\sigma $ exists.
On the interval $[x_i,x_{i+1}]$, $\sigma $ is expressed in a unique way under the form
$$\sigma (t)=\sigma_i(t)=\sum _{k=1}^{n_i+1}\alpha _i^k \prod_{l\neq k}\frac{t-t_i^l}{t_i^k-t_i^l}.$$
Since $\|\sigma -f\|\leq \|f\|$ one has $|\alpha _i^k| \leq 2\,\|f\|$ for all $i,k$.
One deduces that there exists $K$, not depending on $x\in V$ and $i$, such that $|\sigma_i(t)-\sigma_i(t')| \leq K\,  |t-t'|$ for all $t,t'\in [x_i,x_{i+1}]$.

ii) Let again $x\in V$ and  $\sigma  =(\sigma_0, \cdots , \sigma _p)  \in \Sigma (x) $ be an optimal solution  of  (\ref{theta}). Let $y\in V$. We build 
 $\tau  =(\tau_0, \cdots , \tau_p) \in \Sigma (y)$ as follows
\begin{eqnarray*}
\tau_0(t)&=&\sigma_0(t) \qquad \forall \, t\in [y_0,y_1],\\
\tau_1(t)&=&\sigma_1(t)+\tau_0(y_1)-\sigma _1 (y_1) \quad \forall \, t\in [y_1,y_2],\\
\tau_2(t)&=&\sigma_2(t)+\tau_1(y_2)-\sigma  _2(y_2) \quad \forall \, t\in [y_2,y_3],\\
\cdots & & \cdots \cdots \cdots \cdots \cdots \cdots \cdots \cdots \cdots \cdots \cdots \cdots
\end{eqnarray*}
 From what we deduce that for all $t\in [y_i,y_{i+1}]$ one has
 $$\tau_i(t)-\sigma_i(t)=\sum_{k=0}^{i-1}\sigma _k(y_{k+1})-\sigma _{k+1}(y_{k+1}).$$
 
 Recall that  $\sigma _k(x_{k+1})-\sigma _{k+1}(x_{k+1}) =0$ for all $k$. Then, it follows from part i), that there exists $L$ such that 
$| (\tau_i-\sigma_i)(t)| \leq L \, |y_i-x_i |$ for all $t\in [y_i,y_{i+1}]$.   Combining with other intervals one obtains for all $t\in [a,b]$
$$|(\tau  -f)(t)| \leq | (\sigma  -f)(t)| +|( \tau -\sigma )(t)| \leq  | (\sigma  -f)(t)| +L\, \|x-y\|_1,$$
where, for $v\in \R^{p+2}$, $\|v\|_1=\sum_{i=0}^{p+1}|u_i|$.
Hence,
$$\theta (y)\leq \|\tau-f\| \leq \|\sigma-f\| + L \, \|y-x\|_1= \theta (x)+L \, \|y-x\| _1.$$
Finally, permuting the roles playing by $x$ and $y$,
$$|\theta (y)- \theta (x)|\leq L \, \|y-x\| _1 \quad \forall \, x,y\in V.$$ 
Thus,
$\theta $ is Lipschitz on $V$. \end{proof}

 \vspace{0.3cm}

 It is clear that if $(x,\sigma )$ is  a (global) local optimal solution of (\ref{free}), then $\theta (x) =\|\sigma -f\|$.
Since the set $\{(x,\sigma )\, :\,  x\in X, \, \sigma \in \Sigma ( x) \}$  is not convex,
the problem  ($\ref{free}$) is non convex. This leads to  investigate     local optimality.

 \subsection{Local w-minimality}
Through this  subsection,  $x \in X$ and  $\sigma \in \Sigma (x) $   is  such that $\theta (x)=\|\sigma -f\|$.  A point $t$ for which 
$\theta (x) = | \sigma(t)-f(t) |$ is said to be  extreme. 

We are interested in  the existence of local moves    around $x$  which potentially induce a decrease of $\theta $.

 Apply  Theorem \ref{thm:alternance} to the function $\sigma -f$. Let $\varepsilon, k$, $t^-_j, t^+_j$ be  as in the theorem. Apply the algorithm for the construction of points $\xi_j$ given in the last section.

Since $ \sigma $ is optimal, the algorithm stops in the fourth case with  $ i^0\geq 0$  and $j^0$ such that
$$r_{i^0}=n_{i^0}, \quad  x_{i^0}\leq t^+_{j^0}<t^-_{j^0+1}\leq  x_{i^0+1},\quad  r_l\leq n_l \;\; \forall \, l=0,1, \cdots, i^0.$$
 Let $i^-$ be the smallest $i\leq i^0$ such that $r_{k}= n_{k}$ for all $k\in [i,i^0]$.  \\
\noindent Next, let $i^+$ be the greatest $i\geq i^0$ such that 
$$\card(J(l,l+1))\geq  n_l\quad
 \forall \,l=i^0, i^0+1, \cdots, i^+.$$
Then,
$$i^-\leq i^0\leq i^+, \;\;\card(J(l,l+1))\geq n_l \;\,\forall \, l=i^-,i^-+1,\cdots ,i^+,$$
$$\card(J(i^-,i^+ +1))\geq 2+n_{i^-}+n_{i^-+1}+\cdots +n_{i^+} .$$
Let $j^-$ be the smallest $j$ such that  $x_{i^-} \leq t^+_j$. \\
\noindent Let $j^+$ be the greatest $j$ such that  $t^-_j\leq x_{i^++1}$. 
Then,
$$x_{i^-} \leq t^+_{j^-}<t^-_{j^+}\leq x_{i^++1}, \quad j^+\geq  j^-+1+n_{i^-}+n_{i^-+1}+\cdots +n_{i^+}.$$

\begin{pro}
Assume that  $y\in X$ is such that  
$$x_{i^- }\leq y_{i^-}\leq  t^+_{j^-},  \quad t^-_{j^+}\leq y_{i^++1}\leq  x_{i^++1}$$  and 
 $y_i=x_i$ for all $i$ such that $t^+_{j^-}  \leq x_i \leq t^-_{j^+}$.
Then, $\theta (y)\geq \theta(x)$. 
\end{pro}
\begin{proof}Let us consider the fixed knots spline approximation problem on the interval $[t^+_{j^-}  ,t^-_{j^+}]$ with knots
$t^+_{j^-}, x_{i_-+1}, x_{i_-+2}, \cdots , x_i^+, t^-_{j^+}$ and respective degrees $n_{i_-}, n_{i_-+1}, \cdots , n_{i_+}$. Then, due to Proposition \ref{suff}, the restriction of $\sigma $ to the interval is one optimal spline and therefore for any $y\in X$ and $\tau \in \Sigma (y)$
 $$\sup _{t^+_{j^-}\leq t\leq t^-_{j^+}}  |\tau(t)-f(t)| \geq  \sup _{t\in  [t^+_{j^-},t^-_{j^+} ]}| \sigma (t)-f(t)| =\theta (x).$$
 It follows that $\theta (y)\geq \theta (x)$.
 \end{proof}
 
With regard to this proposition, we focus on the knots $x_i$ which belong to the interval $[t^+_{j^-},t^-_{j^+} ]$.  Our strategy  consists in  seeking if a   small move of a given knot $x_i$ can produce a  decrease of  $\card(J(i^-,i^+ +1))$. Only knots $x_i$ such that  $|\sigma (x_i) -f(x_i)| =\theta (x)$ are relevant in this purpose. We shall describe two cases where such a move does work.
 
 \textbf{i) Move on the right.} 
 
 Let $s\in \{-1,1\}$ be such that $\sigma (x_i) -f(x_i) =s\,\theta (x)$. Assume that there exist $\bar \lambda >0$,       $\delta_0 \in \Pi_0, \delta _1 \in \Pi_1, \cdots, \delta _{i-1} \in \Pi_{i-1}$ such that 
  \begin{description}
\item[1.]  $s\, \delta _{i-1}(x_i)>0$,  
\end{description} 
and for $k=0, 1, \cdots, i-1$,
  \begin{description}
 \item[2.] $\theta (x)\geq  | \sigma _k (t)-\lambda \delta_k(t)-f(t) |$ for all  $t\in [x_k,x_{k+1}] $, for all $\lambda \in [0,\bar \lambda]$,
\item[3.] $\delta _k(x_{k+1})=\delta _{k+1}(x_{k+1})$.
\end{description}

Such a situation occurs for  \textbf{$i\leq i_0$}.  Indeed, before stopping at $i_0$,  Algorithm 2 has built intermediary points $\xi_j  \in [a,x_{i_0}]$. Report to the construction of the 
 functions  $\delta _i$ in the proof of  Proposition \ref{gamma}. The functions  $s\delta_0,s \delta_1, \cdots, s\delta _{i-1}$ fulfill   conditions 1,2 and 3. In the next subsection, we shall study in detail the existence of such functions in the general case.

  Let $y\in X$ be such that $y_k=x_k$ for all $k\neq i$. Let $y_i \in (x_i, x_{i+1})$ be close to $x_i$. 
  Let $\lambda \in (0,\bar \lambda ] $.  Set   
  $$ \tau_k = \left\{\begin{array}{ccc}\sigma _k  -\,\lambda \delta _k   &  \textrm{if}& k \leq i-1, \\ \sigma_k & \textrm{if} & k\geq i.\end{array}\right.  $$
  The idea is to take, when this is possible,  $y_i\in (x_i, x_{i+1})$  so that the function $\tau =(\tau_0, \cdots , \tau _p)$ belongs to $\Sigma (y)$. This is the case  if and only if
  $$(\sigma _{i-1}-\, \delta _{i-1})(y_i)= \sigma _i(y_i)$$
Let us introduce 
 $$H(t,\lambda)=(\sigma _{i-1}-\sigma _i)(x_i+t)- \,\lambda \delta _{i-1}(x_i+t).$$
 Then, $H(0,0)=0$.  Assume that $(\sigma_{i-1}-\sigma_i)' (x_i)\neq 0$. The implicit function theorem says that there exists a differentiable function $t(.)$ such that   in a neighbourhood of $0$,  
 
    $$0=H(t(\lambda ), \lambda ), \quad  t'(\lambda )= \frac{\delta _{i-1}(x_i+t(\lambda ))}{ (\sigma_{i-1}-\sigma_i -\lambda \delta _{i-1})'(x_i+t(\lambda ))},$$
 $$t(0)=0, \quad \quad t'(0)=\frac{\,\delta _{i-1}(x_i)}{ (\sigma_{i-1}-\sigma_i )' (x_i)}.$$
 Set $y_i=x_i+t(\lambda) $. Then,    $(\sigma _{i-1}- \delta _{i-1})(y_i)= \sigma _i(y_i)$ and $\tau \in \Sigma (y)$.

A move on the right  of $x_i$ means $t(\lambda ) >0$ which is obtained for small values of $\lambda \in (0,\bar \lambda ]$ under the condition $t'(0)>0$, i.e.,
\begin{description}
\item[4.] $s(\sigma_{i-1}-\sigma_i) '(x_i)>0$. 
\end{description}
 Then, by construction, $x_i<y_i$ and
$$  |\tau (t)-f(t)]  \leq  |\sigma (t)-f(t)| \leq \theta (x) \quad \forall \,t \in [a,b]. $$
Moreover, in view of 1,  there exists $x'_i<x_i$ such that    for $\lambda >0$ small enough 
$$ | \tau (t) -f(t) |< \theta (x)  \quad \forall \, t\in [x'_i,y_i].$$  
Summarizing,
$$\{\,t\in [a,b] \, :\,  |\tau (t)-f(t)]  \leq    \theta (x)\}\subset \{\, t\in [a,b] \, :\,  |\sigma  (t)-f(t)]  \leq    \theta (x)\} ,$$
the inclusion being strict.
Therefore,  $\theta (y)\leq \| \tau -f\| \leq \theta (x)$.
 Unlike $x_i$, the new knot $y_i$ is not an extreme point.
The sequence of alternating extreme points is modified with a possible consequence that    $\theta (y)<\theta (x)$. Anyway,  in the case where  $\theta (y)=\theta (x)$, the number of knots which are extreme has decreased.

 In line with this result, we introduce the following definition of weak local optimal: we say that $\theta $ has  not a local w-minimum at $x\in X$ if for any  neighbourhood $V$ of $x$ there exists $y\in V\cap X$  such that either $\theta (y)< \theta (x)$ or $\theta (y)= \theta (x)$ with a  smaller number of extreme knots.

Therefore the definition of of a local w-minimum is as follows.
\begin{defin}
 $\theta $ has  a local w-minimum at $x\in X$ if for any  neighbourhood $V$ of $x$ there exists no $y\in V\cap X$  such that either $\theta (y)< \theta (x)$  and $\theta (y)= \theta (x)$ with a  smaller number of extreme knots.
\end{defin}

\vspace{0.2cm}
 
  \textbf{ii) Move on the left.} 
 
 Here again, let $s\in \{-1,1\}$ be such that $\sigma (x_i) -f(x_i) =s\,\theta (x)$. Assume that there exist $\bar \lambda >0$,
     $\delta _i \in \Pi_{n_i}, \delta _{i+1} \in \Pi_{n_{i+1}}, \cdots, \delta _{p} \in \Pi_{n_p}$ such that 
  \begin{description}
\item[5.]  $s\, \delta _{i}(x_i)>0$,  
\end{description} 
and for $k=i, i+1, \cdots, p$
  \begin{description}
 \item[6.] $\theta (x)\geq  | \sigma _k (t)- \lambda \delta_k(t)-f(t) |$ for all  $t\in [x_k,x_{k+1}] $, for all $\lambda \in [0,\bar \lambda]$,
\item[7.] $\delta _k(x_{k+1})=\delta _{k+1}(x_{k+1})$.
\end{description}

  Let $y\in X$ be such that $y_k=x_k$ for all $k\neq i$. Let $y_i \in (x_{i-1}, x_{i})$ be close to $x_i$. 
  Let $\lambda \in (0,\bar \lambda ] $.  Set   
  $$ \tau_k = \left\{\begin{array}{ccc}\sigma _k  -\lambda \delta _k   &  \textrm{if}& k \geq i, \\ \sigma_k & \textrm{if} & k< i.\end{array}\right.  $$
In order that   $\tau $ belongs to $\Sigma (y)$, one requires   $$(\sigma _{i}-  \lambda \delta _{i})(y_i)= \sigma _{i-1}(y_i).$$
Let us introduce 
 $$H(t,\lambda)=(\sigma _{i}-\sigma _{i-1})(x_i+t)-\lambda \delta _{i}(x_i+t).$$
In case where $(\sigma_{i-1}-\sigma_i)' (x_i)\neq 0$, the implicit function theorem says that there exists a differentiable function $t(.)$ such that   in a neighbourhood of $0$,  
  $$0=H(t(\lambda ), \lambda ), \quad t'(\lambda )= \frac{\delta _{i}(x_i+t(\lambda ))}{ (\sigma_{i}-\sigma_{i-1} -\lambda \delta _i)'(x_i+t(\lambda ))}, $$
   $$t(0)=0, \quad \quad  t'(0)=\frac{\delta _{i}(x_i)}{ (\sigma_{i}-\sigma_{i-1}) '(x_i)}.$$ 
   
   Set $y_i=x_i+t(\lambda) $. Then,   $\tau \in \Sigma (y)$.   The condition   $t'(0)<0$ is necessary for a move on the left ($y_i<x_i$). This is the case when
    \begin{description}
\item[8.] $s(\sigma_{i-1}-\sigma_i) '(x_i)>0$.
\end{description}
 Then, for $\lambda >0$ small enough,   one obtains   $y_i<x_i $  and $\theta (y)\leq \| \tau -f\| \leq \theta (x)$.  Moreover, there exists
 $x'_i>x_i$ such that  
 $ | \tau (t) -f(t) |< \theta (x)$ for all $t\in [y_i,x'_i]$.  
 Unlike $x_i$, the new knot $y_i$ is not an extreme point. Same consequence as for the left move, $\theta $ has  not a local w-minimum at $x\in X$.
  
 We are ready to establish the main result of this subsection.  
\begin{thm} [Necessary condition for optimality]\label{thm1} 
Let $x\in X$  and  $\sigma \in \Sigma (x)$ be such that $\theta (x)=\| \sigma-f\|$. A necessary condition for  local w-minimality of $\theta $ at $x$ is that at each  knot $x_i$ such that $\theta (x_i)=s( \sigma (x_i)-f(x_i) )$  with $s\in \{-1,1\}$ and $ s (\sigma'_{i-1}(x_i)-\sigma _i'(x_i))>0$    the two following conditions hold:

	i) there are no functions $\delta _k, \, k=0,1,  \cdots , i-1$  fulfilling conditions 1,2,3.
	
	ii) there are no functions $\delta _k, \, k=i, i+1, \cdots ,  p$  fulfilling conditions 5,6,7.
	
\end{thm}
 
 We are left with the question of the existence of such functions.

 \subsection{Existence and constructions of functions $\delta_k$}
Let the knot  $x_i$ be such that $\theta  (x)=s(\sigma -f)(x_i)$ with $s\in \{-1,1\}$ and $s(\sigma_{i-1}-\sigma_i) '(x_i)>0$.  
Then, there is some $j$ such that $x_i\in [t^-_j,t^+_j]$.

We shall describe a process which concludes to the existence or the no existence of functions $\delta _k$ such are in the theorem. It is in the same spirit as Algorithm 2.

i)  We start with the existence   of functions $\delta_k, \, k\geq i$.

\begin{algorithm}[H]
    \caption{Step i}
    Let $m_i=\card(J(i,i+1))$. Then $t^-_{j+m_i-1}\leq x_{i+1}<t^-_{j+m_i}$.\;
    \Case{$i=p$. Then $x_{i+1}=x_{p+1}=b$.}{
        \nl \If{$m_p\geq n_p+2$} {\textbf{No existence.}}
        \nl \If{$m_p\leq n_p+1$}{
            Choose  $ \xi_l\in (t^+_{j+l},t^-_{j+l+1})$, $l=0,\cdots, m_i-1$.\;
Take  $\delta _p(t) =s\prod_{l=0}^{m_i-1} (\xi_l-t)$. \;  
The function $ \delta _p$ responds to the question: \textbf{Existence.} }}
\setcounter{AlgoLine}{0}
\Case{$i<p$}{
    \nl \If{$m_i \geq n_i+2$}{\textbf{No existence}}
    \nl    \If{$m_i\leq n_i$ or  ($m_i= n_i +1$ and $x_{i+1}$ is an extreme point)}{
Choose  $ \xi_l\in (t^+_{j+l},t^-_{j+l+1})$, $l=0,\cdots, m_i-2$. \;
Take $\xi_{m_i-1}=x_{i+1}$.\;
 Take $\delta_i(t) =s\prod_{l=0}^{m-1} (\xi_l-t)$ and $\delta _k(t)=0$ for $k\geq i+1$.\;
 The functions $ \delta _k$ respond to the question: \textbf{Existence.}}
\nl In all other cases go to step $i+1$.}
\end{algorithm}

\begin{algorithm*}[H]
    \caption{Step i+1}
One has $m_i=n_i+1$. Let $m_{i+1}=\card(J(i,i+2))$.
Then $t^-_{j-1+m_{i+1}}\leq x_{i+2}<t^-_{j+m_{i+1}}$\;
\Case{$i=p-1$. Then $x_{i+2}=x_{p+1}=b$.}{
    \nl  \If{$m_{i+1}\geq 2+n_i+ n_{i+1}$}{\textbf{No existence.}}
    \nl \If{$m_{i+1}\leq 1+n_i+n_{i+1}$}{
 Choose  $ \xi_l\in (t^+_{j+l},t^-_{j+l+1})\cap (x_i,x_{i+1})$, $l=0,\cdots, n_i-1$,\;
Choose  $ \xi_l\in (t^+_{j+l},t^-_{j+l+1})\cap (x_{i+1},x_{i+2})$, $l=n_i,\cdots, m_{i+1}-1$.\;
Such choices are possible. Next,
take  $\delta _{p-1}(t) =s\prod_{l=0}^{n_i-1} (\xi_l-t)$,  \;
take  $\delta _{p}(t) =\lambda \prod_{l=n_i}^{m_{i+1}-1} (\xi_l-t)$,  \;
with $\lambda $ taken so that $\delta _{p-1}(x_{p}) =\delta _{p}(x_{p})$,\;
The functions $ \delta _{p-1}, \delta_{p}$ respond to the question: \textbf{Existence.}}}
\setcounter{AlgoLine}{0}
\Case{$i<p-1$}{
    \nl \If{$m_{i+1}\geq n_i+n_{i+1}+2$}{\textbf{No existence}}
    \nl    \If{$m_{i+1}\leq n_i+n_{i+1}$ or  ($m_{i+1}= n_i+n_{i+1}
        +1$ and $x_{i+2}$ is an extreme point).}{
Choose  $ \xi_l\in (t^+_{j+l},t^-_{j+l+1})\cap (x_i,x_{i+1})$, $l=0,\cdots, n_i-1$,\;
Choose  $ \xi_l\in (t^+_{j+l},t^-_{j+l+1})\cap (x_{i+1},x_{i+2})$, $l=n_i,\cdots, m_{i+1}-2$\;
Take $\xi_{m_{i+1}-1}=x_{i+2}$\;
Take  $\delta _i(t) =s\prod_{l=0}^{n_i-1} (\xi_l-t)$,  \;
Take  $\delta _{i+1}(t) =\lambda \prod_{l=n_i}^{m_{i+1}-1} (\xi_l-t)$,  
\;
with $\lambda $ taken so that $\delta _i(x_{i+1}) =\delta _{i+1}(x_{i+1})$\;
Take $\delta _k(t)=0$  for $k\geq i+2$\;
The functions $ \delta _k, \, k\geq i $ respond to the question:  \textbf{Existence.} }
\nl In all other cases go to step $i+2$. }
 \end{algorithm*}
 
\begin{algorithm}[H]
    \caption{Step $i+2$}
One has $m_i=n_i+1$,  $m_{i+1}=n_i+n_{i+1}+1$. \; 
Let $m_{i+2}=\card(J(i,i+3))$.  Then $t^-_{j-1+m_{i+2}}\leq x_{i+3}<t^-_{j+m_{i+2}}$\;
 Continue in the same way as above
\end{algorithm}

   ii)  A symmetric process  allows to conclude to the existence or non existence of functions $\delta _k$ fulfilling  conditions 1 to 3.

 \subsection{Another necessary condition for local optimality and connection with existing results}
Combining the last two subsections, we obtain a reformulation of Theorem~\ref{thm1}.
\begin{thm}\label{thm:52} Let $x\in X$  and  $x\in \Sigma (x)$ be such that $\theta (x)=\| \sigma-f\|$. A necessary condition for  local w-minimality of $\theta $ at $x$ is that at if there is at least one knot  $x_j,~i\in I$, such that $\theta (x)= s(\sigma (x_j)-f(x_j) ), \, s\in \{\-1,1\}$ and $ s(\sigma'_{j-1}(x_j)-\sigma _j'(x_j))>0$  then there exists at least one index $i\in I$, such that the  following two conditions hold.
\begin{enumerate}
\item There exists $k>i$ such that $\card (J(i,k)) \geq 2+n_i+\cdots + n_k$  and,  for  all $l$ such that $i<l<k$,
	 $\card (J(i,l)) \geq 1+n_i+\cdots + n_l$. 	
	\item 
	There exists $k<i$ such that $\card (J(k,i)) \geq 2+n_k+\cdots + n_{i-1}$ and,  for  all $l$ such that $k<l<i$, $\card (J(l,i)) \geq 1+n_l+\cdots + n_{i-1}$.
	\end{enumerate}	
\end{thm}

  Let us place this theorem among the recent  results on the question. In  \cite{Su1,Su4},  a knot $x_i$  such that $\theta (x)= s(\sigma (x_i)-f(x_i) ), \, s\in \{\-1,1\}$ is called unstable  if   $s(\sigma_{i-1}-\sigma_i) '(x_i)>0$, neutral if   $s(\sigma_{i-1}-\sigma_i) '(x_i)=0$ and stable if  $s(\sigma_{i-1}-\sigma_i) '(x_i)<0$.  The formulations of the main theorems in \cite{Su1,Su4,SuU2017} and the present paper are very close except that local optimality is considered in the sense of inf-stationarity in  the sense of Demyanov and Rubinov \cite{dr1,dr2} and in the  local w-minimality in this paper. These two notions are not equivalent, but both aim at describing necessary conditions for optimality. Remark that the situation is a little more general  in \cite{Su4} since the value of the spline may be  fixed at some   knots.
  
  The approach in \cite{Su1,Su4,SuU2017} is analytic since  based on  a minimality criterion   at $x\in X$ using  quasidifferentiability in the sense of Demyanov  \cite{dr1,dr2}. The approach in this paper is  of a constructive type, it is based on the construction of  better  local candidates 
  at minimality. However, the two different approaches together participates to a better understanding of this difficult problem.

It is valuable to note that Theorem~\ref{thm1} and Theorem~\ref{thm:52} do not contain any reference to neutral knots (that is, $s(\sigma_{i-1}-\sigma_i)'(x_i)=0$). At the same time, these knots play an essential role for detecting inf-stationarity~\cite{SuU2017}. Therefore, one of our future research directions is to investigate the connection between neutral knots and the reduction of the number of extreme deviation knots.

\end{document}